\documentclass[10pt, reqno]{amsart}

\usepackage{amsmath,amssymb,amsthm,amsfonts,verbatim}
\usepackage{microtype}
\usepackage[all,2cell]{xy}
\usepackage{mathtools}
\usepackage{graphicx}
\usepackage{pinlabel}
\usepackage{hyperref}
\usepackage{mathrsfs}
\usepackage{color}
\usepackage{enumerate}
\usepackage{cite}

\usepackage[top=1.3in,bottom=1.9in,left=1.2in,right=1.2in]{geometry}

\theoremstyle{plain}
\newtheorem{theorem}{Theorem}[section]
\newtheorem{maintheorem}{Theorem}

\newtheorem{proposition}[theorem]{Proposition}
\newtheorem{lemma}[theorem]{Lemma}

\newtheorem{corollary}[theorem]{Corollary}
\newtheorem{claim}[theorem]{Claim}

\theoremstyle{definition}
\newtheorem{definition}[theorem]{Definition}

\newtheorem{remark}[theorem]{Remark}

\newcommand{\nc}{\newcommand}
\nc{\dmo}{\DeclareMathOperator}

\nc{\Q}{\mathbb{Q}}
\nc{\F}{\mathbb{F}}
\nc{\R}{\mathbb{R}}
\nc{\Z}{\mathbb{Z}}
\nc{\C}{\mathbb{C}}
\nc{\Ell}{\mathcal{L}}
\nc{\M}{\mathcal{M}}
\nc{\K}{\mathcal{K}}
\nc{\I}{\mathcal{I}}
\nc{\U}{\mathcal U}
\nc{\disk}{\mathbb{D}}
\nc{\hyp}{\mathbb{H}}

\nc{\CP}{\mathbb{CP}}
\nc{\cS}{\mathcal{S}}
\dmo{\Mod}{Mod}
\dmo{\PMod}{PMod}
\dmo{\Diff}{Diff}
\dmo{\Homeo}{Homeo}
\dmo{\dist}{dist}
\dmo\BDiff{BDiff}
\dmo\SO{SO}
\dmo\Hom{Hom}
\dmo\SL{SL}
\dmo\Sp{Sp}
\dmo\rank{rank}
\dmo\sig{sig}
\dmo\Out{Out}
\dmo\Aut{Aut}
\dmo\Inn{Inn}
\dmo\GL{GL}
\dmo\PSL{PSL}
\dmo\BHomeo{BHomeo}
\dmo\EHomeo{EHomeo}
\dmo\EDiff{EDiff}
\nc\Sig{\Sigma}
\dmo\Teich{Teich}
\dmo\Fix{Fix}
\nc{\pair}[1]{\langle #1 \rangle}
\nc{\abs}[1]{\left| #1 \right|}
\nc{\action}{\circlearrowright}
\nc{\norm}[1]{\left | \left | #1 \right | \right |}
\nc{\abcd}[4]{\left(\begin{array}{cc} #1 & #2 \\ #3 & #4 \end{array}\right)}
\nc{\into}\hookrightarrow
\dmo{\Isom}{Isom}
\nc{\normal}{\vartriangleleft}
\dmo{\Vol}{Vol}
\dmo{\im}{Im}
\dmo{\Push}{Push}
\dmo{\Conf}{Conf}
\dmo{\PConf}{PConf}
\dmo{\id}{id}
\dmo{\Jac}{Jac}
\dmo{\Pic}{Pic}
\dmo{\Stab}{Stab}
\dmo{\Arf}{Arf}
\dmo{\End}{End}
\dmo{\PB}{PB}
\dmo{\CRS}{CRS}
\nc{\Span}[1]{\operatorname{Span}(#1)}

\renewcommand{\epsilon}{\varepsilon}
\renewcommand{\tilde}{\widetilde}
\nc{\coloneq}{\mathrel{\mathop:}\mkern-1.2mu=}
\nc{\margin}[1]{\marginpar{\scriptsize #1}}
\nc{\para}[1]{\medskip\noindent\textbf{#1.}}
\nc{\red}[1]{\textcolor{red}{#1}}

\title{The Birman exact sequence does not virtually split}

\author{Lei Chen \and Nick Salter}

\email{chenlei@math.uchicago.edu \and salter@math.harvard.edu}

\date{April 30, 2018}

\begin{document}
\maketitle

	\begin{abstract}
	This paper answers a basic question about the Birman exact sequence in the theory of mapping class groups. We prove that the Birman exact sequence does not admit a section over any subgroup $\Gamma$ contained in the Torelli group with finite index. {\em A fortiori} this proves that there is no section of the Birman exact sequence for any finite-index subgroup of the full mapping class group. This theorem was announced in a 1990 preprint of G. Mess, but an error was uncovered and described in a recent paper of the first author.
	\end{abstract}
\section{Introduction}
Let $S$ be a Riemann surface of finite type. A fundamental tool in the study of the mapping class group $\Mod(S)$ of $S$ is the {\em Birman exact sequence} which describes the relationship between $\Mod(S)$ and the mapping class group $\Mod(S')$ of a surface $S'$ obtained from $S$ by filling in boundary components and/or punctures on $S$. In its most basic form, $S = \Sigma_{g,*}$ is a surface of genus $g \ge 2$ with a single puncture $* \in \Sigma_g$, and $S' = \Sigma_g$ is the closed surface obtained by filling in $*$. In this case, the Birman exact sequence takes the form
\begin{equation}\label{equation:BESdef}
1 \to \pi_1(\Sigma_g,*) \to \Mod(\Sigma_{g,*}) \to \Mod(\Sigma_g) \to 1.
\end{equation}

Given any such short exact sequence of groups $1 \to A \to B \to C \to 1$ determined by a surjective homomorphism $f: B \to C$, it is a basic question to determine whether the sequence {\em splits}: that is, whether there is a (necessarily injective) homomorphism $g: C \to B$ such that $f \circ g = \id_C$. In the context of the Birman exact sequence, this question has a topological interpretation: (\ref{equation:BESdef}) can be viewed as the short exact sequence on (orbifold) fundamental groups induced by the fibration $\mathcal M_{g,*} \to \mathcal M_g$ of the ``universal curve'' $\mathcal M_{g,*}$ over the moduli space of Riemann surfaces $\mathcal M_g$. The question of whether (\ref{equation:BESdef}) splits is equivalent\footnote{For brevity's sake, we are ignoring the complications induced by the orbifold structure; of course, these issues disappear when passing to suitably-chosen finite covers of $\mathcal M_g$.} to asking whether the universal curve $\mathcal M_{g,*}$ admits a continuous {\em section}: that is, whether it is possible to continuously choose a marked point on every Riemann surface of genus $g$. 

The Birman exact sequence (\ref{equation:BESdef}) does not split for any $g \ge 2$. This is an easy consequence of two observations. For one, it is easy to construct non-cyclic torsion subgroups of $\Mod(\Sigma_g)$, while it is also simple to show that no such subgroups exist in $\Mod(\Sigma_{g,*})$. However, this argument is somewhat unsatisfactory in that it does not address the more fundamental issue of {\em virtual splitting}. A short exact sequence $1 \to A \to B \to C \to 1$ is said to {\em virtually split} if there exists some finite-index subgroup $C' \le C$ and a homomorphism $g: C' \to B$ such that $f \circ g = \id_{C'}$. The mapping class group $\Mod(\Sigma_g)$ is virtually torsion-free, i.e. there exist finite-index subgroups $\Gamma \le \Mod(\Sigma_g)$ that are torsion-free. Thus for any such $\Gamma$, the argument above breaks down. Formulated in terms of moduli spaces, this leaves a very basic question unanswered: does there exist some finite-sheeted cover $\tilde{\mathcal M_g}$ of $\mathcal M_g$, over which it is possible to find a continuous section of the (pullback of the) universal curve?

For $g=2$ the Birman exact sequence {\em does} virtually split. This follows from the fact that every Riemann surface of genus 2 is hyperelliptic, and hence equipped with $6$ necessarily distinct Weierstrass points. The monodromy of these Weierstrass points is the full symmetric group, but by passing to the $6$-sheeted cover associated with the subgroup $S_5 \le S_6$, one of these points becomes globally distinguishable and hence the universal curve virtually has a section.

The purpose of this note is to show that a similar phenomenon cannot occur for higher genus Riemann surfaces. For the definition of the Torelli group $\mathcal I(\Sigma_g)$, see Section \ref{section:torelli}. 

\begin{maintheorem}\label{theorem:main}
For $g \ge 4$, the Birman exact sequence does not virtually split. Moreover, for any subgroup  $\Gamma \le \mathcal I(\Sigma_g)$ of finite index in the Torelli group, there is no splitting $\sigma: \Gamma \to \mathcal I(\Sigma_{g,*})$ of the Birman exact sequence restricted to $\Gamma$. 
\end{maintheorem}

From the topological point of view, it is natural to consider the more general notion of a multisection of a fiber bundle. A {\em multisection} is a continuous choice of $n$ distinct points on each fiber. For instance, the Weierstrass points form a multisection of cardinality $6$ of the universal curve in genus $2$. As in that particular example, a multisection of a fiber bundle $E \to B$ always induces a genuine section of the pullback bundle over some finite-sheeted cover $B' \to B$. We thus obtain Theorem \ref{maintheorem:corollary} below as an immediate corollary of Theorem \ref{theorem:main}. The {\em Torelli space} is the cover $\mathscr{I}_g \to \mathcal M_g$ of $\mathcal M_g$ corresponding to the subgroup $\mathcal I_g \le \Mod_g$; the universal curve $\mathcal M_{g,*}$ pulls back to give the universal family of ``homologically framed curves'' $\mathscr I_{g,*} \to \mathscr I_g$.

\begin{maintheorem}\label{maintheorem:corollary}
For $g \ge 4$, the universal family $\mathscr I_{g,*} \to \mathscr I_g$ does not admit any continuous multisection. {\em A fortiori}, for $g \ge 4$, the universal curve $\mathcal M_{g,*} \to \mathcal M_g$ does not admit any continuous multisection.
\end{maintheorem}

There is an important bibliographical comment to be made. Theorem \ref{theorem:main} is claimed in the 1990 preprint \cite{mess} of G. Mess. Unfortunately, as detailed in the paper \cite{lei} of the first author, Mess' argument contains a fatal error. In \cite{lei}, the first author proves Theorem \ref{theorem:main} in the special case of the full Torelli group $\Gamma = \mathcal I(\Sigma_g)$. The methods therein make essential use of some special relations in $\mathcal I(\Sigma_g)$ which disappear upon passing to finite-index subgroups. 

In the present note, we return to the outline of the argument as proposed by Mess. In brief, Mess uses the hypothesis of a splitting to construct a particular homomorphism (the homomorphism $s$ constructed and analyzed in Section \ref{section:proof}). Mess incorrectly assumes $s$ to be valued in a certain subgroup of the codomain, and derives a contradiction predicated on this assumption. Our argument proceeds by studying $s$ and deriving a contradiction as Mess does, but a more exhaustive analysis must be carried out.

The paper is organized as follows. Section \ref{section:MCG} collects the necessary facts from the theory of mapping class groups, and establishes some preliminary results. The proof of Theorem \ref{theorem:main} is then carried out in Section \ref{section:proof}.

\section{Mapping class groups}\label{section:MCG}

	\subsection{Canonical reduction systems}
The central tool for the proof of Theorem \ref{theorem:main} is the notion of a {\em canonical reduction system}, which can be viewed as an enhancement of the Nielsen-Thurston classification. We remind the reader that a curve $c\subset S$ is said to be {\em peripheral} if $c$ is isotopic to a boundary component or puncture of $S$. The Nielsen-Thurston classification asserts that each nontrivial element $f\in\Mod(S)$ is of exactly one of the following types: {\em periodic}, {\em reducible}, or {\em pseudo-Anosov}. A mapping class $f$ is {\em periodic} if $f^n = \id$ for some $n \ge 1$, and is {\em reducible} if for some $n \ge 1$, there is some nonperipheral simple closed curve $c \subset S$ such that $f^n(c)$ is isotopic to $c$. If neither of these conditions are satisfied, $f$ is said to be {\em pseudo-Anosov}. In this case, $f$ is isotopic to a homeomorphism $f'$ of a very special form. We will not need to delve into the theory of pseudo-Anosov mappings, and refer the interested reader to \cite[Chapter 13]{FM} and \cite{FLP} for more details. 

	\begin{definition}[{\bf Reduction systems}]
		A \emph{reduction system} of a reducible mapping class $h$ in $\text{\normalfont Mod}(S)$ is a set of disjoint nonperipheral curves that $h$ fixes as a set up to isotopy. A reduction system is \emph{maximal} if it is maximal with respect to inclusion of reduction systems for $h$. The \emph{canonical reduction system} $\text{\normalfont CRS}(h)$ is the intersection of all maximal reduction systems of $h$. 
	\end{definition}
	
	Canonical reduction systems allow for a refined version of the Nielsen-Thurston classification. For a reducible element $f$, there exists $n$ such that $f^n$ fixes each element in CRS$(f)$ and after cutting out CRS$(f)$, the restriction of $f^n$ on each component is either identity or pseudo-Anosov. See \cite[Corollary 13.3]{FM}. In Propositions \ref{CRS1} - \ref{PAcentral}, we list some properties of the canonical reduction systems that will be used later.
	\begin{proposition}\label{CRS1}
		$\text{\normalfont CRS}(h^n)$=$\text{\normalfont CRS}(h)$ for any $n$.
	\end{proposition}
	\begin{proof}This is classical; see  \cite[Chapter 13]{FM}.
	\end{proof}
	
	For two curves $a,b$ on a surface $S$, let $i(a,b)$ be the geometric intersection number of $a$ and $b$. For two sets of curves $P$ and $Q$, we say that $P$ and $Q$ \emph{intersect} if there exist $a\in P$ and $b\in Q$ such that $i(a,b)\neq 0$. We emphasize that ``intersection'' here refers to the intersection of curves on $S$, and not the abstract set-theoretic intersection of $P$ and $Q$ as sets.	
	\begin{proposition}\label{CRS2}
		Let $h$ be a reducible mapping class in $\text{\normalfont Mod}(S)$. If $\{\gamma\}$ and $\text{\normalfont CRS}(h)$ intersect, then no power of $h$ fixes $\gamma$.
		\label{noin}
	\end{proposition}
	
	\begin{proof}
		Suppose that $h^n$ fixes $\gamma$. Therefore $\gamma$ belongs to a maximal reduction system $M$. By definition, $\CRS(h)\subset M$. However $\gamma$ intersects some curve in CRS$(f)$; this contradicts the fact that $M$ is a set of disjoint curves. 
	\end{proof}
	\begin{proposition}\label{CRS3}
		Suppose that $h,f\in \text{\normalfont Mod}(S)$ and $fh=hf$. Then $\text{\normalfont CRS}(h)$ and $\text{\normalfont CRS}(f)$ do not intersect.
		\label{CRS(x,y)}
	\end{proposition}
	\begin{proof}
		Conjugating, CRS($hfh^{-1})=h(\CRS(f))$. Since $hfh^{-1}=f$, it follows that CRS$(f)=h(\CRS(f))$. Therefore $h$ fixes the whole set CRS$(f)$. There is some $n \ge 1$ such that $h^n$ fixes all curves element-wise in CRS$(f)$. By Proposition \ref{noin}, curves in CRS$(h)$ do not intersect curves in CRS$(f)$.
	\end{proof}
	
	For a curve $a$ on a surface $S$, denote by $T_a$ the Dehn twist about $a$. More generally, a {\em Dehn multitwist} is any mapping class of the form
	\[
	T : = \prod T_{a_i}^{k_i}
	\]
	for a collection of pairwise-disjoint simple closed curves $\{a_i\}$ and arbitrary integers $k_i$. 
	
	\begin{proposition}\label{CRSmulti}
	Let 
	\[
	T : = \prod T_{a_i}^{k_i}
	\]
	be a Dehn multitwist. Then
	\[
	\CRS(T) = \{a_i\}.
	\]
	\end{proposition}
	\begin{proof} Firstly $T$ cannot contain any simple closed curves $b$ for which $i(b, a_i) \ne 0$, since no power of $T$ preserves $b$. This can be seen from the equation
	\[
	i(\prod T_{a_i}^{k_i}(b),b)=\sum |k_i|i(a_i,b)\neq 0=i(b,b);
	\]
	see \cite[Proposition 3.2]{FM}. It follows that if $S$ is any reduction system for $T$, then $S \cup \{a_i\}$ is also a reduction system, and hence that $\{a_i\} \subset \CRS(T)$. If $\gamma$ is disjoint from each element of $\{a_i\}$ but not equal to any $a_i$, then there exists some curve $\delta$, also disjoint and distinct from each $a_i$, such that $i(\gamma,\delta) \ne 0$. As both $\{a_i\} \cup \{\gamma\}$ and $\{a_i\} \cup \{\delta\}$ are reduction systems for $T$, this shows that no such $\gamma$ can be contained in $\CRS(T)$ and hence that $\CRS(T) = \{a_i\}$ as claimed. \end{proof}

	The final result we will require follows from the theory of canonical reduction systems. It appears as \cite[Theorem 1]{mccarthy1982normalizers}.
	\begin{proposition}[McCarthy]\label{PAcentral}
	Let $S$ be a Riemann surface of finite type, and let $f \in \Mod(S)$ be a pseudo-Anosov element. Then the centralizer subgroup of $f$ in $\Mod(S)$ is virtually cyclic.
	\end{proposition}

	\subsection{The Torelli group, separating twists, and bounding pair maps}\label{section:torelli}
	The action of a diffeomorphism $f$ on the homology of a surface $S$ is well-defined on the level of isotopy, giving rise to the {\em symplectic representation}
	\[
	\Psi:\Mod(S) \to \Aut(H_1(S;\Z)).
	\]
	The {\em Torelli group} is the kernel subgroup $\mathcal I(S) := \ker(\Psi)$.  There are several classes of elements of the Torelli group that will feature in the proof of Theorem \ref{theorem:main}. For context, background, and proofs of the following assertions, see \cite[Chapter 6]{FM}. A {\em separating twist} is a Dehn twist $T_c$, where $c$ is a separating curve on $S$. Separating twists $T_c \in \mathcal I(S)$ are elements of the Torelli group. A pair of curves $\{a,b\} \subset S$ is said to be a \emph{bounding pair} if $a,b$ are individually nonseparating, but $a \cup b$ bounds a subsurface of $S$ of positive genus on both sides. A {\em bounding pair map} is the Dehn multitwist $T_a T_b^{-1}$; necessarily $T_a T_b^{-1} \in \mathcal I(S)$ for any bounding pair $\{a,b\}$.

	\subsection{Point- and disk-pushing subgroups} Recall the Birman exact sequence (\ref{equation:BESdef}). The kernel $\pi_1(\Sigma_g,*)$ is referred to as the {\em point-pushing subgroup} of $\Mod(\Sigma_{g,*})$. Henceforth we will tidy our notation and omit reference to the basepoint. An element $\alpha \in \pi_1(\Sigma_g)$ determines a mapping class $\alpha \in \Mod(\Sigma_{g,*})$ as follows: one ``pushes'' the marked point $*$ along the loop determined by $\alpha$.
	
	There is an analogous notion of a ``disk-pushing subgroup''. Let $S = \Sigma_{g,1}$ denote a surface of genus $g$ with one boundary component. In this setting, the Birman exact sequence takes the form
	\begin{equation}\label{equation:BESdisk}
	1 \to \pi_1(UT\Sigma_g) \to \Mod(\Sigma_{g,1}) \to \Mod(\Sigma_g) \to 1.
	\end{equation}	
	Here, $UT\Sigma_g$ denotes the {\em unit tangent bundle} of $\Sigma_g$; i.e. the $S^1$-subbundle of the tangent bundle $T\Sigma_g$ consisting of unit-length tangent vectors (relative to an arbitrarily-chosen Riemannian metric). In this context, the kernel $\pi_1(UT\Sigma_g)$ is known as the {\em disk-pushing subgroup}. An element $\tilde \alpha \in \pi_1(UT\Sigma_g)$ determines a ``disk-pushing'' diffeomorphism of $\Sigma_{g,1}$ as follows: one treats the boundary component $\Delta$ as the boundary of a disk $D$, and ``pushes'' $D$ along the path determined by the image $\alpha \in \pi_1(\Sigma_g)$. The extra information of the tangent vector encoded in $\tilde \alpha$ is used to give a consistent framing of $\partial D$ along its path.
	
	The proposition below records some basic facts about point- and disk-pushing subgroups. In item \ref{item:pushCRS} below, the {\em support} of a (not necessarily simple) element $\alpha \in \pi_1(\Sigma_g)$ is defined to be the minimal subsurface $S_\alpha \subset \Sigma_{g,*}$ that contains $\alpha$ for which every component of $\partial S_\alpha$ is {\em essential}, i.e. non-nullhomotopic and nonperipheral.
	\begin{proposition}\label{lemma:pushfacts}\ 
	\begin{enumerate}
	\item\label{item:torelli} There are containments $\pi_1(\Sigma_g) \le \mathcal I(\Sigma_{g,*})$ and $\pi_1(UT\Sigma_g) \le \mathcal I(\Sigma_{g,1})$.
	\item Let $\alpha \in \pi_1(\Sigma_g)$ be a {\em simple} element. Viewed as a point-push map, $\alpha$ has an expression as a bounding pair map
	\[
	\alpha = T_{\alpha_L} T_{\alpha_R}^{-1},
	\]
	where $\alpha_L, \alpha_R$ are the simple closed curves on $\Sigma_{g,*}$ lying to the left (resp. right) of $\alpha$. 
	\item Let $\zeta \in \pi_1(UT\Sigma_g)$ be a generator of the kernel of the map $\pi_1(UT\Sigma_g) \to \pi_1(\Sigma_g)$. Viewed as a push map, $\zeta = T_\Delta$, the twist about the boundary component of $\Sigma_{g,1}$.
	\item\label{item:tildealpha} Let $\tilde \alpha \in \pi_1(UT\Sigma_g)$ be simple (in the sense that $\alpha \in \pi_1(\Sigma_g)$ is simple). Viewed as a disk-pushing map, there is an expression
	\[
	\tilde \alpha = T_{\alpha_L} T_{\alpha_R}^{-1} T_\Delta^k
	\]
	for some $k \in \Z$. 
	
	\item\label{item:pushCRS} Let $\alpha \in \pi_1(\Sigma_g)$ be an arbitrary (not necessarily simple) element. Then 
	\[
	\CRS(\alpha) = \partial(S_\alpha),
	\]
	the (possibly empty) boundary of the support $S_\alpha$. Moreover, $\alpha$ is pseudo-Anosov on the subsurface $S_\alpha$.
	\end{enumerate}
	\end{proposition}
	\begin{proof}
	Items \eqref{item:torelli}- \eqref{item:tildealpha} are standard; see \cite[Chapters 4,6]{FM} for details. Item \eqref{item:pushCRS} is a reformulation of a theorem of Kra, adapted to the language of canonical reduction systems. See \cite[Theorem 14.6]{FM}.
	\end{proof}

	In Section \ref{section:proof}, we will make use of the following lemma concerning the action of separating twist maps on the underlying fundamental group.
	\begin{lemma}\label{lemma:BPint}
	Let $T_c \in \mathcal I(\Sigma_{g,*})$ be a Dehn twist about a separating simple closed curve $c$. Let $\alpha \in \pi_1(\Sigma_g)$ be an arbitrary element, represented as a (not necessarily simple) curve based at $* \in \Sigma_{g,*}$. If
	\[
	T_c^k(\alpha) = \alpha
	\]
	for any $k \ne 0$, then there exists a representative of $\alpha$ that is disjoint from $c$. 
	\end{lemma}
	\begin{proof}
	The hypothesis implies that $T_c^k$ and $\alpha$ commute as elements of $\mathcal I(\Sigma_{g,*})$. By Propositions \ref{CRS3}, \ref{CRSmulti}, and \ref{lemma:pushfacts}.\ref{item:pushCRS}, $\CRS(\alpha) = \partial(S_\alpha)$ and $\CRS(T_c^k) = \{c\}$ must be disjoint, and moreover 
	\[
	\{c\} \subset \Sigma_{g,*} \setminus S_\alpha.
	\]
	The result follows. 
	\end{proof}
	
	\subsection{Lifts of some special mapping classes} The foundation of the proof of Theorem \ref{theorem:main} is an analysis of the possible images of bounding pair maps and separating twists under a hypothetical section. Let $\Gamma \le \mathcal I(\Sigma_g)$ be a finite-index subgroup, and suppose that $\sigma: \Gamma \to \Mod(\Sigma_{g,*})$ is a section. A first observation is that in fact, $\sigma(\Gamma) \le \mathcal I(\Sigma_{g,*})$. This follows readily from the fact that $\pi_1(\Sigma_g) \le \mathcal I(\Sigma_{g,*})$ as observed in Proposition \ref{lemma:pushfacts}.\ref{item:torelli}.
	
	Since $\Gamma$ is a finite-index subgroup of $\mathcal I(\Sigma_g)$, there is no assumption that a given separating twist $T_c$ or bounding pair map $T_aT_b^{-1}$ is an element of $\Gamma$. However, the assumption that $\Gamma$ is of finite index in $\mathcal I(\Sigma_g)$ does imply that for each separating twist $T_c$, and each bounding pair map $T_a T_b^{-1}$, there is some $k > 0$ (depending on the individual element) such that $T_c^k \in \Gamma$, and likewise $T_a^k T_b^{-k} \in \Gamma$.

In the following lemma and throughout, for a curve $\tilde c$ on $\Sigma_{g,1}$ (resp. $\Sigma_{g,*}$), when we say $\tilde c$ is isotopic to a curve $c$ on $\Sigma_g$, we mean that $\tilde c$ is isotopic to $c$ after forgetting the puncture (resp. boundary component).
	
	\begin{lemma}\label{lemma:liftfacts}\ 
	\begin{enumerate}
	\item \label{item:BP} Let $\{a,b\}$ be a bounding pair, and fix $k>0$ such that $(T_aT_b^{-1})^k\in \Gamma$. Up to a swap of $a$ and $b$, we have that $\sigma((T_aT_b^{-1})^k)=(T_{a'}T_{b'}^{-1})^k(T_{a'}^{-1}T_{a''})^{n}$, where $n$ is an integer and $a',a'',b'$ are three disjoint curves on $\Sigma_{g,1}$ such that $a',a''$ are isotopic to $a$ and $b'$ is isotopic to $b$. Notice that $n$ can be zero.
	
	\item \label{item:sep} Let $c$ be a separating curve on $\Sigma_g$ that divides $\Sigma_g$ into two subsurfaces each of genus at least two. For any $k >0$ such that $(T_c)^k\in \Gamma$, we have that $\sigma((T_c)^k)=(T_{c'})^{k}(T_{c'}^{-1}T_{c''})^{n}$ where $n$ is an integer and $c'$ and $c''$ are a pair of curves on $\Sigma_{g,1}$ that are both isotopic to $c$.
	\end{enumerate}
	\end{lemma}
	\begin{proof}
		Let $(T_aT_b^{-1})^k\in \Gamma$ be a power of a bounding pair map. Since the centralizer of $(T_aT_b^{-1})^k$ contains a copy of $\mathbb{Z}^{2g-3}$ as a subgroup of $\mathcal I(\Sigma_g)$, the centralizer of $(T_aT_b^{-1})^k$ as a subgroup of $\Gamma$ contains a copy of $\mathbb{Z}^{2g-3}$ as well. By the injectivity of $\sigma$, the centralizer of $\sigma(T_aT_b^{-1})\in \mathcal I(\Sigma_{g,*})$ contains a copy of $\mathbb{Z}^{2g-3}$. When $g>3$, we have that  $2g-3>3$. Therefore $\sigma((T_aT_b^{-1})^k)\in \mathcal I(\Sigma_{g,*})$ cannot be pseudo-Anosov because the centralizer of a pseudo-Anosov element is a virtually cyclic group by Proposition \ref{PAcentral}. For any curve $\gamma'$ on $\Sigma_{g,*}$, denote by $\gamma$ the same curve on $\Sigma_g$. We decompose the proof into the following three steps.
		
		\begin{claim}[\bf Step 1]
			CRS$(\sigma((T_aT_b^{-1})^k))$  only contains curves that are isotopic to $a$ or $b$.
		\end{claim}
		\begin{proof}
			
			Suppose that there exists $\gamma' \in \CRS(\sigma((T_aT_b^{-1})^k))$ such that $\gamma$ is not isotopic to $a$ or $b$. There are two cases.
			
			{\bf Case 1: $\gamma$ intersects $a$ or $b$.} Since a power of $\sigma((T_aT_b^{-1})^k)$ fixes $\gamma'$, a power of $(T_aT_b^{-1})^k$ fixes $\gamma$. On the other hand, CRS$((T_aT_b^{-1})^k)=\{a,b\}$. Combined with Lemma \ref{noin}, this shows that $(T_aT_b^{-1})^k$ does not fix $\gamma$. This is a contradiction.
			
			{\bf Case 2: $\gamma$ does not intersect $a$ and $b$.} In this case by the change-of-coordinates principle, there exists a separating curve $c$ on $\Sigma_g$ such that $i(a,c)=0$, $i(b,c)=0$ and $i(c,\gamma)\neq 0$.  Assume that $T_c^m\in \Gamma$. Since $(T_aT_b^{-1})^k$ and $T_c^m$ commute in $\Gamma$, the two mapping classes $\sigma((T_aT_b^{-1})^k)$ and $\sigma(T_c^m)$ commute in $\mathcal I(\Sigma_{g,*})$. Therefore a power of $\sigma(T_c^m)$ fixes CRS($\sigma(T_aT_b^{-1})$); more specifically a power of $T_c^m$ fixes $\gamma$. However by Lemma \ref{noin}, no power of $T_c$ fixes $\gamma$. This is a contradiction.
		\end{proof}
		
		\begin{claim}[\bf Step 2]
	CRS$(\sigma((T_aT_b^{-1})^k))$ must contain curves $a'$ and $b'$ that are isotopic to $a$ and $b$, respectively.
		\end{claim}
		\begin{proof}
			Suppose that CRS$(\sigma((T_aT_b^{-1})^k))$ does not contain a curve $a'$ isotopic to $a$. Then by Step 1, CRS$(\sigma((T_aT_b^{-1})^k))$ either contains one curve $b'$ isotopic to $b$ or two curves $b'$ and $b''$ both isotopic to $b$. After cutting $\Sigma_{g,*}$ along CRS$(\phi((T_aT_b^{-1})^k))$, there is some component $C$ that is not a punctured annulus. $C$ is homeomorphic to the complement of $b$ in $\Sigma_g$.
			
By the Nielsen-Thurston classification, a power of $\sigma((T_aT_b^{-1})^k)$ is either pseudo-Anosov on $C$ or else is the identity on $C$. If a power of $\sigma((T_aT_b^{-1})^k)$ is pseudo-Anosov on $C$,  then the centralizer of $\sigma((T_aT_b^{-1})^{k})|_C$ is virtually cyclic by Proposition \ref{PAcentral}. Combining with $T_{b'}$ and $T_{b''}$,  the centralizer of $\sigma((T_aT_b^{-1})^k)$ in $\mathcal I(\Sigma_{g,*})$ is virtually an abelian group of rank at most $3$. This contradicts the fact that the centralizer of $\sigma((T_aT_b^{-1})^k)$ contains a subgroup $\mathbb{Z}^{2g-3}$, since $g \ge 4$ and hence  $2g-3>3$. Therefore $\sigma((T_aT_b^{-1})^k)$ is the identity on $C$. However, viewing $C = \Sigma_g \setminus \{b\}$ as a subsurface of $\Sigma_g$ that contains $a$, we see that $(T_aT_b^{-1})^k$ is actually not the identity on $C$; this is a contradiction.
		\end{proof}
		
		\begin{claim}[\bf Step 3]
			$\sigma((T_aT_b^{-1})^k)=(T_{a'}T_{b'}^{-1})^k(T_{a'}^{-1}T_{a''})^{n}$, where $n$ is an integer and $a',a'',b'$ are three disjoint curves on $\Sigma_{g,*}$ such that $a',a''$ are isotopic to $a$ and $b'$ is isotopic to $b$. 		\end{claim}
		\begin{proof}
			Suppose that $\sigma((T_aT_b^{-1})^k)$ is pseudo-Anosov on some component $C$ of 
			\[
			\Sigma_{g,*} \setminus \CRS(\sigma((T_aT_b^{-1})^k))
			\]
			Since the genus $g(C)\ge 1$, there exists a separating curve $s$ on $C$ such that $\sigma(T_s^m)$ commutes with $\sigma((T_aT_b^{-1})^k)$ in $\sigma(\Gamma)$. Therefore, some power of $\sigma((T_aT_b^{-1})^k)$ fixes CRS($\sigma(T_s^m))$, which is either one curve or two curves isotopic to $s$. Thus a power of $\sigma((T_aT_b^{-1})^k)$ fixes some curve on $C$, which means that $\sigma((T_aT_b^{-1})^k)$ is not pseudo-Anosov on $C$. It follows that a power of $\sigma((T_aT_b^{-1})^k)$ must be a product of Dehn twists about the curves in CRS$(\sigma((T_aT_b^{-1})^k))$. Since $\sigma((T_aT_b^{-1})^k)$ is a lift of $(T_aT_b^{-1})^k$, the lemma holds.
		\end{proof}
		The same argument works for $T_c^m\in \Gamma$ the Dehn twist about a separating curve $c$ as long as both components of $\Sigma_g\setminus \{c\}$ have genus two or greater.			
		\end{proof}
		
\subsection{The handle-pushing subgroup}
As in Mess's approach, we will prove Theorem \ref{theorem:main} by showing that certain ``handle-pushing'' subgroups (contained in any finite-index subgroup of $\mathcal I(\Sigma_g)$) do not admit sections to $\mathcal I(\Sigma_{g,*})$. To define these, let $c$ be a separating curve. The complement $\Sigma_g\setminus\{c\}=\Sigma_{p,1}\cup \Sigma_{q,1}$ is a disconnected surface with two components. Let $\mathcal I(c)\le \mathcal I(\Sigma_g)$ be the subgroup consisting of Torelli mapping classes that are a product of mapping classes with supports on either $\Sigma_{p,1}$ or $\Sigma_{q,1}$. The subgroup $\mathcal I(c)$ satisfies the following exact sequence:
	\[
	1\to \mathbb{Z}\to \mathcal I(\Sigma_{p,1})\times \mathcal I(\Sigma_{q,1})\to \mathcal I(c)\to 1,
	\]
	where $\mathbb{Z}$ is generated by $T_c$.

\begin{definition}[Handle-pushing subgroup]
Let $c$ be a separating curve as in Figure \ref{figure:handlepush}, dividing $\Sigma_g \setminus \{c\} = \Sigma_{p,1} \cup \Sigma_{q,1}$. The {\em handle-pushing subgroup} on $\Sigma_{p,1}$, written $\mathcal H(\Sigma_{p,1})$, is defined as
\[
\mathcal H(S) := \pi_1( UT\Sigma_{p}) \le \mathcal I(c).
\]
More broadly, any finite-index subgroup of $\mathcal H(\Sigma_{p,1})$ will also be called a handle-pushing subgroup.
\end{definition}

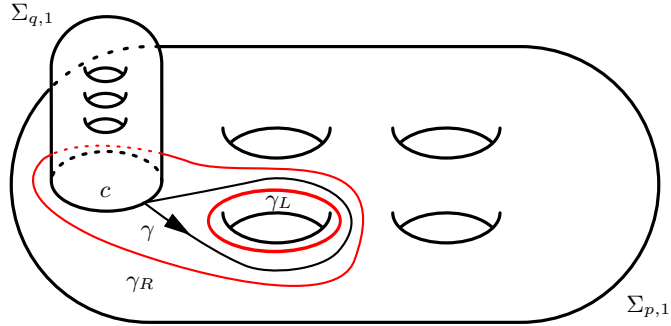
\begin{figure}[h]
\labellist
\small
\pinlabel $c$ [b] at 36 44.8
\pinlabel $\Sigma_{p,1}$ [tl] at 230.4 12.8
\pinlabel $\Sigma_{q,1}$ [br] at 19.2 110.4
\pinlabel $\gamma$ [tr] at 56.8 39.2
\pinlabel $\gamma_L$ [t] at 101.6 51.4
\pinlabel $\gamma_R$ [tr] at 56.8 20
\endlabellist
\includegraphics{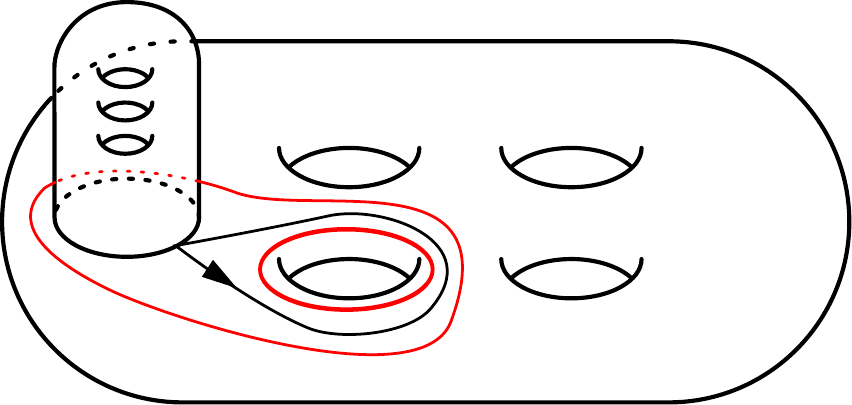}
\caption{An element of a handle-pushing subgroup.}
\label{figure:handlepush}
\end{figure}

\begin{remark}\label{remark:Hprops}
Every finite-index subgroup of $\mathcal H(\Sigma_{p,1})$, being isomorphic to a finite-index subgroup of $\pi_1 (UT\Sigma_{p})$, is isomorphic to a non-split extension of a surface group of genus $p' \ge p$ by $\Z$. 
\end{remark}
	
	Denote by  $A\le \mathcal I(c)$ the group generated by the disk-pushing subgroup on both subsurfaces $\Sigma_{p,1}$ and $\Sigma_{q,1}$. Then $A$ satisfies the following exact sequence:
	\begin{equation}
	1\to \mathbb{Z}\to \pi_1(UT\Sigma_p)\times \pi_1(UT\Sigma_q)\xrightarrow{p} A\to 1
	\label{UTB}
	\end{equation}

	\begin{lemma}
		The exact sequence \eqref{UTB} does not virtually split.
		\label{UTBNS}
	\end{lemma}
	\begin{proof}
		This will be proved via group cohomology. For a $\mathbb{Z}$-central extension of a group $T$ 
		\begin{equation}
		1\to \mathbb{Z}\to \widetilde{T}\xrightarrow{\alpha} T\to 1,
		\label{gce}
		\end{equation}
		there is an associated {\em Euler class} $Eu(\alpha)\in H^2(T;\mathbb{Z})$. The extension $\alpha$ splits if and only if $Eu(\alpha)$ vanishes; see \cite[Chapter 4.3]{brown}. $Eu(\alpha)$ can be constructed using the Lyndon-Hochschild-Serre spectral sequence of \eqref{gce}, by taking $Eu(\alpha)=d_2(1)$. Here $d_2$ is the differential $d_2:\mathbb{Z}\to H^2(T;\mathbb{Z})$ on the $E_2$ page. The (rational) Betti number $b_1(\tilde T)$ can be computed from the spectral sequence as
		\[
		b_1(\widetilde{T})=b_1(T)+\dim(\ker( d_2)).
		\] 
		Therefore $Eu(\alpha)\ne 0$ is nonvanishing if and only if $b_1(\widetilde{T})=b_1(T)$.
		
		Let $A'\xrightarrow{i}A$ be a finite-index subgroup of $A$. Let $\widetilde{A'}=p^{-1}(A')$. The goal is to prove that the top row of the diagram 
		\begin{equation}
		\xymatrix{
		1\ar[r] 
		&\mathbb{Z}\ar[r]\ar[d] &\widetilde{A'}\ar[r]^{\beta}\ar[d] &A'\ar[r]\ar[d] &1\\
			1\ar[r]& \mathbb{Z}\ar[r]&\widetilde{A}\ar[r]^{\alpha} &A\ar[r] &1}
			\label{eu}
		\end{equation}
		does not split. It suffices to show that $Eu(\beta)\neq 0\in H^2(A';\mathbb{Q})$.  By the theory of the transfer homomorphism,
		\[
		i^*:H^2(A;\mathbb{Q})\to H^2(A';\mathbb{Q})
		\]
		 is injective. By construction, $Eu(\beta)=i^*(Eu(\alpha))$. Therefore it suffices to establish that 
		 \[
		 Eu(\alpha)\neq 0\in H^2(A;\mathbb{Q}).
		 \] 
		 By the above discussion, we only need to show that $b_1(A)=b_1(\pi_1(UT\Sigma_p)\times \pi_1(UT\Sigma_q))$. However, since $p \ge 2$ and $q \ge 2$ by assumption, 
		 \[
		 b_1(\pi_1(UT\Sigma_p)\times \pi_1(UT\Sigma_q))=b_1(\pi_1(\Sigma_p)\times \pi_1(\Sigma_q)).
		 \]
		  Since $A\to  \pi_1(\Sigma_p)\times \pi_1(\Sigma_q)$ is surjective, it follows that $b_1(A)\ge b_1(\pi_1(\Sigma_p)\times \pi_1(\Sigma_q))$, and so $b_1(A)=b_1(\pi_1(UT\Sigma_p)\times \pi_1(UT\Sigma_q))$ as desired.
		  \end{proof}
	
As a corollary, we can refine the analysis of $\sigma(T_c^k)$ for $T_c$ a separating twist, as begun in Lemma \ref{lemma:liftfacts}.\ref{item:sep}. 
	
	\begin{lemma}\label{lemma:seplift}
		Let $c \subset \Sigma_g$ be a separating curve such that each component of $\Sigma_g \setminus\{c\}$ has genus at least $2$, and let $k>0$ be such that $T_c^k\in \Gamma$. Then there exists a curve $\tilde c \subset \Sigma_{g,*}$ isotopic to $c$ such that $\sigma(T_c^k)=T_{\tilde c}^k$.
	\end{lemma}
	\begin{proof}
		If this is not the case, then $\sigma(T_c^k)=T_{c'}^lT_{c''}^m$ where $c',c''$ bound an annulus and $l\neq 0,m\neq 0$. Let $A$ be the subgroup constructed above, relative to the separating curve $c$. The image $\sigma(A\cap \Gamma)$ must be contained in the centralizer of $T_{c'}^lT_{c''}^m$. In particular, $\sigma(A\cap \Gamma)$ must be contained in the disk-pushing subgroups on the sides of $c'$ and $c''$ not bounding the annulus. This gives a virtual splitting of exact sequence \eqref{UTB}, contradicting Lemma \ref{UTBNS}.
		\end{proof}

\section{Proof of Theorem \ref{theorem:main}}\label{section:proof}

\para{Beginning the proof}  Let $\Gamma \le \mathcal I(\Sigma_g)$ be a subgroup of finite index, and suppose that $\sigma: \Gamma \to \mathcal I(\Sigma_{g,*})$ is a section. By the hypothesis that $g \ge 4$, there exists a separating simple closed curve $c \subset \Sigma_g$ that divides $\Sigma_g$ into subsurfaces $\Sigma_{p,1}$ and $\Sigma_{q,1}$ with $p,q \ge 2$. Let $T_c$ denote the corresponding Dehn twist. Choosing $k$ such that $T_c^k \in \Gamma$, Lemma \ref{lemma:seplift} asserts that $\sigma(T_c^k) = T_{\tilde c}^k$ for some separating curve $\tilde c \subset \Sigma_{g,*}$. Without loss of generality, we assume that the marked point $* \in \Sigma_{p,1}$. 

A standard argument using canonical reduction systems shows that $\sigma(\Mod(\Sigma_{p,1}) \cap \Gamma)$ is supported on the subsurface $\tilde{\Sigma_{p,1}} \cong \Sigma_{p,1,*}$ bounded by $\tilde c$. The Birman exact sequences for $\Sigma_{p,1}$ and $\Sigma_{p,1,*}$ (restricted to the Torelli group\footnote{The Torelli group is not unambiguously defined for a surface $\Sigma_{p,1,*}$. The meaning here of $\mathcal I(\Sigma_{p,1,*})$ is simply the full preimage $\pi^{-1}(\mathcal I(\Sigma_{p,1}))$.}) fit together in the following commutative diagram, where the group $\PB_{1,1}(\Sigma_p)$ and the homomorphism $p_{2,*}$ will be described below.
\begin{equation}\label{equation:BES}
\xymatrix{
1 \ar[r]	& \PB_{1,1}(\Sigma_p) \ar[r] \ar[d]_{p_{2,*}} 			& \mathcal I(\Sigma_{p,1,*}) \ar[r] \ar[d]_\pi				& \mathcal I(\Sigma_p) \ar@{=}[d] \ar[r]	&1\\
1 \ar[r]	& \pi_1(UT\Sigma_p) \ar[r] \ar@{-->}@/_1pc/[u]_{\sigma}	& \mathcal I(\Sigma_{p,1}) \ar[r] \ar@{-->}@/_1pc/[u]_{\sigma}	& \mathcal I(\Sigma_p) \ar[r]	&1
}
\end{equation}
The group $\PB_{1,1}(\Sigma_p)$ is defined as the fundamental group of the configuration space $\PConf_{1,1}(\Sigma_p)$, where
\[
\PConf_{1,1}(\Sigma_p) := \{(x,v) \mid x \in \Sigma_p,  v \in T^1_y(\Sigma_p),\ x \ne y \}.
\]
Here, $T^1_y(\Sigma_p)$ denotes the space of unit-length tangent vectors in the tangent space $T_y(\Sigma_p)$, relative to an arbitrarily-chosen Riemannian metric. Projection onto either factor realizes $\PConf_{1,1}(\Sigma_p)$ as a fibration in two ways:
\begin{equation}\label{equation:fibrations}
\xymatrix{
					& \Sigma_{p,*} \ar[d]							&	\\
UT(\Sigma_{p,*})\ar[r]	& \PConf_{1,1}(\Sigma_p) \ar[r]^-{p_1} \ar[d]_{p_2}	& \Sigma_p\\
					& UT(\Sigma_p)		
}
\end{equation}

\subsection{The map $s$} We now come to the central object of study in the argument. Let $\mathcal H = \mathcal H(\Sigma_{p,1}) \cap \Gamma$ denote the handle-pushing subgroup. Combining diagrams (\ref{equation:BES}) and (\ref{equation:fibrations}), we obtain a homomorphism
\begin{equation}\label{equation:tildes}
\tilde s := p_{1,*} \circ \sigma: \mathcal H \to \pi_1(\Sigma_p). 
\end{equation}
We will see that $\tilde s$ has paradoxical properties, leading to a contradiction that establishes the non-existence of the section $\sigma$. A first observation is that we can replace $\tilde s$ by a map between surface groups. Let $\varpi: \pi_1(UT\Sigma_p) \to \pi_1(\Sigma_p)$ denote the projection, and define $\overline{\mathcal H} := \varpi(\mathcal H)$. By construction $\overline{\mathcal H}$ is a finite-index subgroup of $\pi_1(\Sigma_p)$.  

\begin{lemma}\label{lemma:sbar}
There is a homomorphism
\begin{equation}\label{equation:s}
s: \overline{\mathcal H} \to \pi_1(\Sigma_p)
\end{equation}
such that $\tilde s$ factors as $\tilde s = s \circ \varpi$. 
\end{lemma}

\begin{proof}
As noted in Remark \ref{remark:Hprops}, $\mathcal H$ has the structure of a cyclic central extension of a finite-index subgroup $\overline{\mathcal H} \le \pi_1(\Sigma_p)$. Viewed as a subgroup of $\mathcal I(\Sigma_{p,1})$, the center of $\mathcal H$ consists of elements of the form $T_c^k$. As already observed, $\sigma(T_c^k) = T_{\tilde c}^k$, where $\tilde c$ is the boundary of the subsurface $\Sigma_{p,1,*} \le \Sigma_{g,*}$. The map $p_{1,*}: \PB_{1,1}(\Sigma_p) \to \pi_1(\Sigma_p)$ is induced from the boundary-capping map $\Sigma_{p,1,*} \to \Sigma_{p,*}$. The result follows.
\end{proof}

The construction of $s$ allows us to continue the analysis of $\sigma$ begun in Lemma \ref{lemma:liftfacts}, giving a complete description of $\sigma$ on (powers of) bounding-pair maps. 

\begin{lemma}\label{lemma:BPlifts}
Let $a,b$ form a bounding pair on $\Sigma_g$. Then there exists a bounding pair $\tilde a, \tilde b$ on $\Sigma_{g,*}$ such that $\sigma(T_a^k T_b^{-k}) = T_{\tilde a}^k T_{\tilde b}^{-k}$ for any $k$ such that $T_a^k T_b^{-k} \in \Gamma$. 
\end{lemma}

\begin{proof}
Let $c$ be a separating curve on $\Sigma_g$ dividing $\Sigma_g$ into components $\Sigma_{p,1}, \Sigma_{q,1}$, each of genus $p,q \ge2$. Let $a,b$ be a bounding pair on $\Sigma_g$ such that $a,b,c$ forms a pair of pants; observe that for {\em any} bounding pair $a,b$, there exists a curve $c$ as above. For instance, in Figure \ref{figure:handlepush}, the curves $\{\gamma_L,\gamma_R\}$ form such a bounding pair relative to the $c$ as shown. As $T_a^k T_b^{-k}$ commutes with $T_c^\ell$, the same is true for the lifts $\sigma(T_a^k T_b^{-k})$ and $\sigma(T_c^\ell) = T_{\tilde c}^\ell$. In particular, $\sigma(T_a^k T_b^{-k})$ is supported on exactly one component of the surface $\Sigma_{g,*} \setminus \{\tilde c\}$. There are thus two possibilities to consider, depending on whether this component also contains $*$. 

According to Lemma \ref{lemma:liftfacts}.\ref{item:BP}, there are simple closed curves $\tilde a, \tilde a', \tilde b \subset \Sigma_{g,*}$ and   an integer $m$ such that
\begin{equation}\label{equation:charlie}
\sigma(T_a^k T_b^{-k}) = T_{\tilde a}^{k-m} T_{\tilde a'}^{m} T_{\tilde b}^{-k}.
\end{equation}
The curves $\tilde a$ and $\tilde a'$ are isotopic on $\Sigma_g$, but may not be isotopic on $\Sigma_{g,*}$, i.e. $\tilde a \cup \tilde a'$ can bound an annulus $A$ containing the marked point $*$. If this is not the case, then $\tilde a, \tilde a'$ determine the same isotopy class on $\Sigma_{g,*}$, and the result follows. Note that in the case where $A$ and $*$ are contained in distinct components of $\Sigma_{g,*} \setminus\{\tilde c\}$, this must necessarily hold. 

We therefore assume that $A$ and $*$ are contained in the same component $\Sigma_{p,1,*} \subset \Sigma_{g,*}$. Since $a,b,c$ form a pair of pants on $\Sigma_g$, it follows that $T_a^k T_b^{-k} \in \mathcal H$, the handle-pushing subgroup. In fact, there is a {\em one-to-one correspondence} between elements of $\overline{\mathcal H}$ represented by (a power of) a {\em simple} closed curve on $\Sigma_{p}$, and the set of bounding pairs $a,b$ under consideration. We write $\alpha(a,b) \in \pi_1(\Sigma_p)$ for the element of $\overline{\mathcal H}$ corresponding to the bounding pair $T_aT_b^{-1}$. Our proof now proceeds by analyzing $s$ on such elements of $\overline{\mathcal H}$. 

As observed above, $*$ may or may not be contained in the annulus $A$. If $*$ is not, we can reformulate the above argument by observing that $s(\alpha(a,b)^k) = 1$. In the remaining case, we aim to show that either $m=0$ or $m = k$ in \eqref{equation:charlie}. As (without loss of generality) $\tilde a'$ becomes isotopic to $\tilde b$ upon capping $c$ by a disk, it follows that 
\begin{equation}\label{equation:BPlift}
s(\alpha(a,b)^k) = p_{1,*}(T_{\tilde a}^{k-m} T_{\tilde a'}^{m} T_{\tilde b}^{-k}) = T_{\tilde a}^{k-m} T_{\tilde b}^{m-k} = \alpha(a,b)^{m-k}.
\end{equation}

To summarize, we have shown that for all bounding pairs $a,b$ under consideration, there is an integer $m(a,b,k)$ such that 
\[
s(\alpha(a,b)^k) = \alpha(a,b)^{m(a,b,k)}.
\]
The desired assertion $m = 0$ or $m = k$ now follows from Lemma \ref{lemma:Ghom} below. \end{proof}

\begin{lemma}\label{lemma:Ghom}
Let $G \le \pi_1(\Sigma_p)$ be a subgroup of finite index, and let $f: G \to \pi_1(\Sigma_p)$ be an arbitrary homomorphism. Suppose that for all simple elements $\alpha \in \pi_1(\Sigma_p)$, there is an integer $m(\alpha,k)$ such that 
\[
f(\alpha^k) = \alpha^{m(\alpha,k)}.
\]
Then either $m(\alpha,k) = 0$ or else $m(\alpha,k) = k$, independent of $\alpha$.
\end{lemma}
\begin{proof}
Suppose $\alpha, \beta$ are simple elements. Then for any $\ell$, the conjugate $\beta^\ell \alpha \beta^{-\ell}$ is also simple. Choose $k, \ell$ such that $\alpha^k$ and $\beta^\ell$ are both elements of $G$. Then definitionally,
\begin{equation}\label{equation:a1}
f( \beta^\ell \alpha^k \beta^{-\ell}) = (\beta^\ell \alpha \beta^{-\ell})^ {m(\beta^\ell \alpha \beta^{-\ell}, k)}.
\end{equation}
On the other hand, it is clear that $m(\beta,-\ell) = -m(\beta, \ell)$, and so
\begin{equation}\label{equation:a2}
f( \beta^\ell \alpha^k \beta^{-\ell})  = f(\beta^\ell) f(\alpha^k) f(\beta^{-\ell}) = \beta^{m(\beta, \ell)} \alpha^{m(\alpha,k)} \beta^{-m(\beta, \ell)}.
\end{equation}
For an arbitrary nontrivial element $\gamma \in \pi_1(\Sigma_p)$ and integers $m,n$, the elements $\gamma^m$ and $\gamma^n$ are conjugate if and only if $m = n$. It follows that $m(\alpha,k) = m(\beta^\ell \alpha \beta^{-\ell}, k)$. Thus,
\[
(\beta^\ell \alpha \beta^{-\ell})^ {m(\alpha, k)} = \beta^{m(\beta, \ell)} \alpha^{m(\alpha,k)} \beta^{-m(\beta, \ell)},
\]
and so
\[
\beta^{\ell - m(\beta,\ell)} \alpha^{m(\alpha,k)} \beta^{m(\beta,\ell) - \ell} = \alpha^{m(\alpha,k)}.
\]
Nontrivial elements $x,y \in \pi_1(\Sigma_p)$ commute if and only if there are nonzero integers $c,d$ such that $x^c = y^d$. As $\alpha, \beta$ were assumed to be simple, we conclude that one of three conditions must hold: (1) $\alpha = \beta^{\pm 1}$, or (2) $\ell = m(\beta,\ell)$ or else (3) $m(\alpha,k) = 0$. 

Case (1) provides no further information; we henceforth assume that $\alpha \ne \beta^{\pm 1}$. To finish the argument, we must show that if $m(\alpha,k) = 0$, then $m(\beta,\ell) = 0$ for all $\beta, \ell$. Suppose to the contrary that there is some $\beta$ such that $m(\beta, \ell) \ne 0$. Reversing the roles of $\alpha$ and $\beta$ in the above argument, we see that (2) must hold and so $k = m(\alpha,k)$, but this contradicts the assumption $m(\alpha,k) = 0$. 
\end{proof}

Translated into the setting of the homomorphism $s:\overline{\mathcal H} \to \pi_1(\Sigma_p)$, Lemmas \ref{lemma:BPlifts} and \ref{lemma:Ghom} combine to give the following immediate but crucial corollary.
\begin{corollary}\label{corollary:cases}
The homomorphism $s:\overline{ \mathcal H} \to \pi_1(\Sigma_p)$ has one of the following properties:
\begin{enumerate}[(A)]
\item $s(\alpha^k) = \alpha^k$ for all elements $\alpha^k \in \overline{\mathcal H}$ such that $\alpha \in \pi_1(\Sigma_p)$ is simple.
\item $s(\alpha^k) = 1$ for all elements $\alpha^k \in \overline{\mathcal H}$ such that $\alpha \in \pi_1(\Sigma_p)$ is simple.
\end{enumerate}
\end{corollary}

The next step of the argument considers cases (A) and (B) separately. In both cases, we will see that the formula defining $s$ on simple elements extends to all of $\overline{\mathcal H}$.

\subsection{Case (A)}
\begin{lemma}\label{lemma:caseA}
Suppose $s$ has property (A) of Corollary \ref{corollary:cases}. Then $s: \overline{\mathcal H} \to \pi_1(\Sigma_p)$ is given by the inclusion map.
\end{lemma}
\begin{proof}
This follows easily from the method of proof of Lemma \ref{lemma:Ghom}. Let $\beta \in \overline{\mathcal H}$ be an arbitrary element, let $\alpha \in \pi_1(\Sigma_p)$ be simple, and let $\alpha^k \in \overline{\mathcal H}$. Then $\beta \alpha \beta^{-1}$ is also simple, and $\beta \alpha^k \beta^{-1} \in \overline{\mathcal H}$. As $\beta \alpha \beta^{-1}$ is simple,
\[
f(\beta \alpha^k \beta^{-1}) = \beta \alpha^k \beta^{-1};
\]
on the other hand,
\[
f(\beta \alpha^k \beta^{-1}) = f(\beta) \alpha^k f(\beta)^{-1}.
\]
Arguing as in Lemma \ref{lemma:Ghom}, this implies $f(\beta) = \beta$ as desired.
\end{proof}

\subsection{Case (B)}\label{subsection:B}
\begin{lemma}\label{lemma:caseB}
 Suppose $s$ has property (B) of Corollary \ref{corollary:cases}. Then $s: \overline{\mathcal H} \to \pi_1(\Sigma_p)$ is the trivial homomorphism.
\end{lemma}
The proof of Lemma \ref{lemma:caseB} will require a further analysis of $s$. This will require some preliminary work to describe. By passing to a further finite-index subgroup $\Gamma' \le \Gamma$ if necessary, we can assume that $\mathcal H \le \pi_1(UT\Sigma_p)$ is {\em characteristic} and hence the conjugation action of $\mathcal I(\Sigma_{p,1})$ on $\pi_1(UT\Sigma_p)$ preserves $\mathcal H$. This descends to an action of $\mathcal I(\Sigma_{p,*})$ on $\overline{\mathcal H}$. Thus there is a homomorphism
\[
\lambda: \mathcal I(\Sigma_{p,*}) \to \Aut(\overline{\mathcal H}).
\]
Consider now the images $\overline{\Gamma} \le \mathcal I(\Sigma_{p,*})$ and $\overline{\overline{\Gamma}} \le \mathcal I(\Sigma_p)$. By construction, $\overline{\Gamma} \cap \pi_1(\Sigma_p) = \overline{\mathcal H}$. As conjugation by $\overline{\mathcal H}$ is an inner automorphism, $\lambda$ descends to a homomorphism 
\[
\overline \lambda: \overline{\overline{\Gamma}} \to \Out(\overline{\mathcal H}).
\]

\begin{lemma}\label{lemma:equivariant}
The homomorphism $s$ is $\overline{\overline{\Gamma}}$-equivariant. That is, for any outer automorphism $[\alpha] \in \overline{\overline{\Gamma}}$ and any $x \in \overline{\mathcal H}$, the conjugacy classes of $s(\alpha\cdot x)$ and $\alpha\cdot s(x)$ in $\pi_1(\Sigma_g)$ coincide. 
\end{lemma}
\begin{proof}
Let $a \in \overline{\overline{\Gamma}}$ be given. Choose an element $\alpha \in \Gamma$ descending to the outer automorphism class $a$. By construction, for $x \in \overline{\mathcal{H}}$, the image $s(x)$ is given by $(p_{1,*} \circ \sigma)(\tilde x)$, where $\tilde x \in \mathcal H$ is any lift. On $\mathcal H$, the action of $\overline{\overline{\Gamma}}$ is induced by the conjugation action $\tilde x \mapsto \alpha \tilde x \alpha^{-1}$. Thus
\[
s(a \cdot x) = p_{1,*}(\sigma(\alpha \tilde x \alpha^{-1})) = p_{1,*}(\sigma(\alpha))\  s(x)\  p_{1,*}(\sigma(\alpha))^{-1}.
\]
Here we exploit the fact that $p_{1,*}: \PB_{1,1}(\Sigma_p) \to \pi_1(\Sigma_p)$ is the restriction of the forgetful homomorphism 
\[
p_{1,*}: \mathcal I(\Sigma_{p,1,*}) \to \mathcal I(\Sigma_{p,*}).
\]
 To finish the argument, it suffices to show that $[p_{1,*}(\sigma(\alpha))] = a$ as elements of $\mathcal I(\Sigma_p)$. This follows from the fact that $\sigma: \Gamma \to \mathcal I(\Sigma_{p,1,*})$ is a section of the map $p_{2,*}: \mathcal I(\Sigma_{p,1,*}) \to \mathcal I (\Sigma_{p,1})$ in combination with the commutativity of the diagram
\[
\xymatrix{
\mathcal I(\Sigma_{p,1,*}) \ar[r]^{p_{1,*}} \ar[d]_{p_{2,*}}		& \mathcal I(\Sigma_{p,*}) \ar[d]\\
\mathcal I(\Sigma_{p,1}) \ar[r]				&\mathcal I(\Sigma_p).
}
\]
\end{proof}

\begin{proof} (of Lemma \ref{lemma:caseB}) Let $x \in \overline{\mathcal H}$ be an arbitrary element, and let $d$ be an arbitrary separating curve on $\Sigma_{p,1}$. Taking $k$ such that $T_d^k \in \Gamma$ and applying Lemma \ref{lemma:equivariant}, there is an equality
\[
s(T_d^k(x)) = T_d^k (s(x))
\]
of conjugacy classes in $\pi_1(\Sigma_p)$. To proceed, we will analyze the conjugacy class of $T_d^k(x)$ in $\overline{\mathcal H}$. This is complicated by the fact that in this expression, $T_d^k$ acts on $x$ {\em not} as a separating twist on $\Sigma_{p}$, but rather as the {\em lift} of such a twist to the finite-sheeted cover $\Sigma_r \to \Sigma_p$ corresponding to the finite-index subgroup $\overline{\mathcal H}$. 

\begin{lemma}\label{BPfactor}
Let $T_d$ be a Dehn twist on $\Sigma_{p,1}$, and let $x \in \overline{\mathcal H}$ be an arbitrary element. Then there exists some $k \ge 1$, {\em simple} elements $\gamma_1, \dots, \gamma_N$ of $\pi_1(\Sigma_p)$ and integers $f_1, \dots, f_N$, such that $\gamma_i^{f_i} \in \overline{\mathcal H}$ for all $i$, and there is an expression
\[
T_d^k(x)  = \gamma_1^{f_1} \dots \gamma_N^{f_N} x
\]
of elements of $\overline{\mathcal H}$. 
\end{lemma}
\begin{proof}
Let $\pi: \Sigma_r \to \Sigma_p$ be the covering map associated to the containment $\overline{\mathcal H} \le \pi_1(\Sigma_p)$. For $k$ sufficiently large, $T_d^k$ lifts to a mapping class on $\Sigma_r$. This lift is not unique, but there is a unique lift up to the action of the deck group of $\pi$. Since $T_d$ is a Dehn twist on $\Sigma_p$, there is a distinguished lift
\begin{equation}\label{multilift}
\tilde{T_d^k} = \prod T_{\tilde {d_i}}^{k_i} 
\end{equation}
of $T_d^k$ as a multitwist on $\Sigma_r$, for certain integers $k_i$. Here, the set $\{\tilde {d_i}\}$ consists of all components of the preimage $\pi^{-1}(d)$. Observe that each curve $\tilde{d_i}$ is contained in the $\pi_1(\Sigma_p)$-conjugacy class of $d^{e_i}$ for some $e_i$, and that also the conjugacy class of $d^{e_i}$ is contained in $\overline{\mathcal H}$. As the deck group is finite, we can assume that $T_d^k$ acts on $\overline{\mathcal H}$ as a genuine multitwist, possibly after further increasing $k$.

Choose representative curves for each $\tilde{d_i}$, and represent $x \in \overline{\mathcal H}$ as a map $x(t): [0,1] \to \Sigma_r$, chosen so as to intersect the set $\{\tilde{d_i}\}$ in minimal position. This determines a sequence of arcs $\alpha_1, \dots, \alpha_{N+1}$ as follows. The points of intersection between $x$ and $\{\tilde{d_i}\}$ can  be enumerated via $0 < t_1 < \dots < t_N < t_{N+1} = 1$ such that $x(t)$ intersects the multicurve $\{\tilde{d_i}\}$ if and only if $t = t_m$ for some $1 \le m \le N$. The arc $\alpha_m$ is then defined as the image of $x$ restricted to the interval $[0,t_m]$ (so in particular, $\alpha_{N+1} = x$).

Each arc $\alpha_m$ connects $*$ to one of the curves $\tilde{d_i}$, and thus determines an element $\gamma_m'$ of $\overline{\mathcal H}$ in the conjugacy class of the appropriate $\tilde{d_i}$. The geometric description of $T_d^k$ as a multitwist allows one to obtain an expression for $T_d^k(x)$ of the desired form. The curve $T_d^k(x)$ can be described as follows: first $T_d^k(x)$ follows $\alpha_1$ to the first point of intersection with $\{\tilde{d_i}\}$; this is the curve corresponding to $\gamma_1'$. Then $T_d^k(x)$ winds around $\gamma_1'$ a number of times $f_1'$ as specified by \eqref{multilift}. Then $T_d^k(x)$ continues along the portion of $\alpha_2$ running from $t = t_1$ to $t=t_2$, and continues, winding around each $\gamma_i'$ some number of times $f_i'$ in succession. 

By construction, after each crossing of $\gamma_m'$, the curve $T_d^k(x)$ traverses the portion of $\alpha_{m+1}$ from $t_m$ to $t_{m+1}$. This can be replaced by first backtracking along $\alpha_m$, and then traversing the entirety of $\alpha_{m+1}$. Written as an element of $\pi_1(\Sigma_r) = \overline{\mathcal H}$, this analysis produces an expression
\[
T_d^k(x)  = \gamma_1'^{f_1'} \dots \gamma_N'^{f_N'} x.
\]
The claim now follows from the observation that each $\gamma_m'$ is a based loop on $\Sigma_r$ corresponding to a curve $\tilde{d_i}$. Each $\tilde{d_i}$ is a component of the preimage of $d$. As an element of $\pi_1(\Sigma_p)$, each $\gamma_m'$ is thus of the form $\gamma_m' = \gamma_m^{e_m}$ for some {\em simple} curve $\gamma_m \in \pi_1(\Sigma_p)$ in the conjugacy class of $d$. Taking $f_m = e_m f_m'$, the result follows. 
\end{proof}

Applying Lemma \ref{BPfactor}, there is an equality
\begin{equation}\label{equation:s}
s(T_c^k(x)) = s(\gamma_1^{f_1} \dots \gamma_N^{f_n} x) = s(\gamma_1^{f_1}) \dots s(\gamma_N^{f_N}) s(x) = s(x),
\end{equation}
with the last equality holding by Corollary \ref{corollary:cases}(B) since all the $\gamma_i$ are simple. We conclude that there is an equality of $\pi_1(\Sigma_g)$-conjugacy classes 
\[
T_c^k(s(x)) = s(x).
\]
By Lemma \ref{lemma:BPint}, this implies that $s(x)$ is disjoint from $c$ as curves on $\Sigma_p$. As this argument applies for {\em every} separating curve on $\Sigma_p$, we conclude that $s(x)$ must be disjoint from every separating curve $c$ on $\Sigma_p$. Since $p \ge 2$, an easy argument with the change-of-coordinates principle implies that any nontrivial element $y \in \pi_1(\Sigma_p)$ must intersect some separating curve $c$. This shows that $s(x)$ must be trivial as claimed.
\end{proof}

\subsection{Finishing the argument} The final stage of the argument exploits the fact that the existence of a section $\sigma: \overline{\mathcal H} \to \PB_{1,1}(\Sigma_p)$ places strong homological constraints on the map $s$. Throughout this section, our cohomology groups will implicitly have rational coefficients. To simplify matters further, we forget the (inessential) tangential data encoded in the space $\PConf_{1,1}(\Sigma_p)$, and consider instead the induced section
\[
\overline{\sigma}: \overline{\mathcal H} \to \PB_2(\Sigma_p);
\]
here $\PB_2(\Sigma_p)= \pi_1(\PConf_2(\Sigma_p))$ is the fundamental group of the configuration space of two ordered points on $\Sigma_p$. The space $\PConf_2(\Sigma_p)$ is, by definition, given as
\[
\PConf_2(\Sigma_p) := \Sigma_p \times \Sigma_p \setminus \Delta,
\]
where $\Delta$ is the diagonal locus. In this setting, there is a factorization
\[
s = p_{2,*} \circ \overline{\sigma}.
\]

A crucial consequence of this is that $s^*: H^*(\Sigma_p) \to H^*(\overline{\mathcal H})$ factors through $H^*(\PB_2(\Sigma_p))$. The following lemma is proved by a standard argument using the formulation of Poincar\'e duality via Thom spaces. 

\begin{lemma}\label{lemma:HPConf}
Let $[\Delta] \in H^2(\Sigma_p \times \Sigma_p)$ denote the Poincar\'e dual class of $\Delta$, and let
\[
\iota: \PConf_2(\Sigma_p) \to \Sigma_p \times \Sigma_p
\]
denote the inclusion map. Then $\iota^*([\Delta]) = 0 \in H^2(\PConf_2(\Sigma_p))$.
\end{lemma}

\para{Concluding the proof} Let $i: \overline{\mathcal H} \to \pi_1(\Sigma_p)$ denote the inclusion. Consider the product homomorphism
\[
i \times s: \overline{\mathcal H} \to \pi_1(\Sigma_p) \times \pi_1(\Sigma_p) \cong \pi_1(\Sigma_p \times \Sigma_p).
\]
Observe that this coincides with the section map $\overline{\sigma}: \overline{\mathcal H} \to \PB_2(\Sigma_g)$, so that there is a factorization
\[
i \times s = \iota \circ \overline{\sigma}.
\]
By Lemma \ref{lemma:HPConf}, it follows that $(i \times s)^*([\Delta]) =  (\iota \circ \overline{\sigma})^*([\Delta]) = 0 \in H^2(\overline{\mathcal H})$.

Let $x_1, y_1, \dots, x_p, y_p \in H^1(\Sigma_p)$ denote a symplectic basis with respect to the cup product form; let also $[\Sigma_p]$ denote the fundamental class. Then 
\[
[\Delta] = 1 \otimes [\Sigma_p] + [\Sigma_p] \otimes 1 + \sum_{i = 1}^p x_i \otimes y_i - y_i \otimes x_i
\]
as a class in 
\[
H^2(\Sigma_p \times \Sigma_p) \cong (H^0(\Sigma_p) \otimes H^2(\Sigma_p)) \oplus (H^2(\Sigma_p) \otimes H^0(\Sigma_p)) \oplus (H^1(\Sigma_p) \otimes H^1(\Sigma_p)).
\]
Thus
\[
0 = (i \times s)^*([\Delta]) = i^*(1) s^*([\Sigma_p]) + i^*([\Sigma_p]) s^*(1) + \sum_{j = 1}^p i^*(x_j) s^*(y_j) - i^*(y_j)s^*(x_j). 
\]

We will see that in both cases (A) and (B), this is a contradiction. Lemma \ref{lemma:caseA} asserts that in Case (A), $s^* = i^*$ in degree $1$. Since $H^*(\Sigma_p)$ is generated as an algebra in degree $1$, this implies that $s^* = i^*$ in degree $2$ as well. Then a basic calculation shows that in this case, 
\[
(i \times s)^*([\Sigma_p]) = (i \times i)^*([\Sigma_p]) = \chi(\Sigma_p)[\overline{\mathcal H}],
\]
where $\chi(\Sigma_p)$ denotes the Euler characteristic and $[\overline{\mathcal H}]$ denotes the fundamental class of the surface group $\overline{\mathcal H}$. As this is nonzero, we have arrived at a contradiction. Similarly, Lemma \ref{lemma:caseB} asserts that in Case (B), $s^* = 0$ in positive degrees. Then $(i\times s)^*([\Sigma_p]) = i^*([\Sigma_p]) = \chi(\Sigma_p)[\overline{\mathcal H}] \ne 0$, again a contradiction.  \qed

    	\bibliography{citing}{}

\begin{thebibliography}{McC82}

\bibitem[Bro94]{brown}
K.~Brown.
\newblock {\em Cohomology of groups}, volume~87 of {\em Graduate Texts in
  Mathematics}.
\newblock Springer-Verlag, New York, 1994.
\newblock Corrected reprint of the 1982 original.

\bibitem[Che17]{lei}
L.~Chen.
\newblock The universal surface bundle over the {T}orelli space has no
  sections.
\newblock Pre-print, https://arxiv.org/abs/1710.00786, 2017.

\bibitem[FLP12]{FLP}
A.~Fathi, F.~Laudenbach, and V.~Po\'enaru.
\newblock {\em Thurston's work on surfaces}, volume~48 of {\em Mathematical
  Notes}.
\newblock Princeton University Press, Princeton, NJ, 2012.
\newblock Translated from the 1979 French original by Djun M. Kim and Dan
  Margalit.

\bibitem[FM12]{FM}
B.~Farb and D.~Margalit.
\newblock {\em A primer on mapping class groups}, volume~49 of {\em Princeton
  Mathematical Series}.
\newblock Princeton University Press, Princeton, NJ, 2012.

\bibitem[McC82]{mccarthy1982normalizers}
J.D. McCarthy.
\newblock Normalizers and centralizers of pseudo-anosov mapping classes.
\newblock Pre-print, 1982.

\bibitem[Mes90]{mess}
G.~Mess.
\newblock Unit tangent bundle subgroups of the mapping class groups.
\newblock MSRI Pre-print, 1990.

\end{thebibliography}
	\bibliographystyle{alpha}

\end{document}